\documentclass{amsart}
\usepackage{amsmath,amsthm,amsfonts,amssymb,amscd,latexsym}

\newtheorem{thm}{Theorem}
\newtheorem{lem}{Lemma}

\newtheorem{rem}{Remark}

\numberwithin{equation}{section}

\newcommand{\Ker}{\mathrm{Ker}}

\newcommand{\End}{\mathrm{End}}

\newcommand{\lie}{\mathrm{Lie}}
\newcommand{\ad}{\mathrm{ad}}
\newcommand{\bi}{\mathrm{bi}}
\newcommand{\hemi}{\mathrm{hemi}}
\newcommand{\K}{\mathrm{Gr}}

\newcommand{\fbb}{\mathbb{F}}
\newcommand{\cbb}{\mathbb{C}}

\newcommand{\lf}{\mathfrak{L}}
\newcommand{\gf}{\mathfrak{g}}
\newcommand{\nf}{\mathfrak{n}}
\newcommand{\slf}{\mathfrak{sl}}

\newcommand{\ovo}{\overline{\otimes}}

\mathchardef\myhyphen="2D 

\begin{document}


\title[Leibniz bimodules and Grothendieck rings]{Tensor products of
Leibniz bimodules and Grothendieck rings -- an addendum}

\author{J\"org Feldvoss}
\address{Department of Mathematics and Statistics, University of South Alabama,
Mobile, AL 36688-0002, USA}
\email{jfeldvoss@southalabama.edu}

\author{Chao Ma}
\address{Independent Researcher, London, UK}
\email{raymond.ma@me.com}

\author{Friedrich Wagemann}
\address{Laboratoire de math\'ematiques Jean Leray, UMR 6629 du CNRS,
Nantes Universit\'e, 2, rue de la Houssini\`ere, F-44322 Nantes Cedex 3, France}
\email{wagemann@math.univ-nantes.fr}

 


\begin{abstract}
We continue our investigation of certain associativity properties (namely,
alternativity, power-associativity, or being a Jordan ring) of the
non-associative multiplication defined on the Grothendieck group of
the category of finite-dimensional Leibniz bimodules which we started
in \cite{FW}. Namely, we prove that for every finite-dimensional non-zero
Leibniz algebra over an arbitrary field this Grothendieck ring is not
power-associative. As a consequence, the Grothendieck ring is also
neither alternative nor a Jordan ring. This corrects and generalizes
some of our results in the last section of \cite{FW}.
\end{abstract}


\date{October 28, 2026}
          
\maketitle


In this note we correct and generalize some of our results in Section 5 of
\cite{FW} about certain associativity properties (namely, alternativity,
power-associativity, or being a Jordan ring) of a Grothendieck ring
associated to the category of finite-dimensional Leibniz bimodules.

For all the undefined terminology we refer the reader for non-associative
algebras to the book of Schafer \cite{S} or the survey paper of Bremner,
Murakami, and Shestakov \cite{BMS}, for Lie algebras to the book of
Humphreys \cite{H}, and for Leibniz algebras to our previous papers
\cite{F1} or \cite{FW}.

The key for all the results in this paper is the following stronger version
of Lemma~5.3\,(c) in \cite{FW}:

\begin{lem}\label{criterion}
Let $A$ and $B$ be unital commutative rings. If there exist elements
$a\in\widetilde{A}$ and $b\in\widetilde{B}$ such that $aa=m1_A+
\sum\limits_{\tilde{a}\in\widetilde{A}}m_{\tilde{a}}\tilde{a}$ and at least
one of the integers $m_{\tilde{a}}$ is non-zero as well as $bb=n1_B
+\sum\limits_{\tilde{b}\in\widetilde{B}}n_{\tilde{b}}\tilde{b}$ for some
non-zero integer $n$, then $A\circledast B$ is not power-associative.
\end{lem}

\begin{proof}
We set $a^2:=a*a=aa$, $u:=\sum\limits_{\tilde{a}\in\widetilde{A}}m_{\tilde{a}}
\tilde{a}$, $b^2:=b*b=bb$, and $v:=\sum\limits_{\tilde{b}\in\widetilde{B}}
n_{\tilde{b}}\tilde{b}$. Then we have that $a^2=m1+u$ and $b^2=n1+v$.

Next, suppose that $a^3:=a^2*a=a*a^2=m'1_A+\sum\limits_{\tilde{a}\in
\widetilde{A}}m_{\tilde{a}}'\tilde{a}$ and $b^3:=b^2*b=b*b^2=n'1_B+
\sum\limits_{\tilde{b}\in\widetilde{B}}n_{\tilde{b}}'\tilde{b}$.\footnote{Note
that $(A,*)$ is a unital subring of the unital commutative product $A\circledast
B$. Hence, $a^2\in A$ implies that $a^3\in A$. Similarly, since $(B,*)$ is a
unital subring of the unital commutative product $A\circledast B$ and
$b^2\in B$, we have that $b^3\in B$.} By setting $u':=\sum\limits_{\tilde{a}
\in\widetilde{A}}m_{\tilde{a}}'\tilde{a}$ and $v':=\sum\limits_{\tilde{b}\in
\widetilde{B}}n_{\tilde{b}}'\tilde{b}$, we obtain that $a^3=m'1+u'$ and
$b^3=n'1+v'$.

In order to show that $A\circledast B$ is not power-associative, we consider
the element $x:=sa+tb\in A\circledast B$ for arbitrary integers $s$ and $t$,
set $x^2:=x*x$, and compute $(x^2*x)*x$ as well as $x^2*x^2$.

Now observe that
$$u*a=(a^2-m1)*a=a^3-ma=m'1+u'-ma$$
and
$$v*b=(b^2-n1)*b=b^3-nb=n'1+v'-nb\,.$$
From these identities and
\begin{eqnarray*}
x^2 & = & (sa+tb)^2=s^2a^2+st(a*b+b*a)+t^2b^2=s^2a^2+t^2b^2\\
& = & s^2(m1+u)+t^2(n1+v)=(s^2m+t^2n)1+s^2u+t^2v\,,
\end{eqnarray*}
we obtain that
\begin{eqnarray*}
x^2*x & = & [(s^2m+t^2n)1+s^2u+t^2v]*(sa+tb)\\
& = & (s^2m+t^2n)sa+(s^2m+t^2n)tb+s^3(u*a)+t^3(v*b)\\
& = & (s^2m+t^2n)sa+(s^2m+t^2n)tb+s^3(m'1+u'-ma)+t^3(n'1+v'-nb)\\
& = & (s^3m'+t^3n')1+st^2na+s^3u'+s^2tmb+t^3v'\,.
\end{eqnarray*}

By employing the identities
$$u'*a=(a^3-m'1)*a=(a^3*a)-m'a$$
and
$$v'*b=(b^3-n'1)*b=(b^3*b)-n'b\,.$$
we deduce that
\begin{eqnarray*}
(x^2*x)*x & = & [(s^3m'+t^3n')1+st^2na+s^3u'+s^2tmb+t^3v']*(sa+tb)\\
& = & (s^3m'+t^3n')sa+s^2t^2na^2+s^4u'*a\\
&& +(s^3m'+t^3n')tb+s^2t^2mb^2+t^4v'*b\\
& = & (s^3m'+t^3n')sa+s^2t^2na^2+s^4(a^3*a)-s^4m'a\\
&& +(s^3m'+t^3n')tb+s^2t^2mb^2+t^4(b^3*b)-t^4n'b\\
& = & s^4(a^3*a)+t^4(b^3*b)+s^2t^2na^2+s^2t^2mb^2+st^3n'a+s^3tm'b\\
& = & s^4(a^3*a)+t^4(b^3*b)+s^2t^2n(m1+u)+s^2t^2m(n1+v)+st^3n'a+s^3tm'b\\
& = & s^4(a^3*a)+t^4(b^3*b)+2s^2t^2mn1+s^2t^2(nu+mv)+st^3n'a+s^3tm'b\,.
\end{eqnarray*}

On the other hand, we have that
$$a^2*a^2=(m1+u)*(m1+u)=m^21+2mu+u*u\,,$$
or equivalently,
$$u^2:=u*u=(a^2*a^2)-m^21-2mu\,.$$
Similarly, we obtain that
$$b^2*b^2=(n1+v)*(n1+v)=n^21+2nv+v*v\,,$$
or equivalently,
$$v^2:=v*v=(b^2*b^2)-n^21-2nv\,.$$
From these identities for $u^2$ and $v^2$ in conjunction with $u*v=0=v*u$, it
follows that
\begin{eqnarray*}
x^2*x^2 & = & [(s^2m+t^2n)1+s^2u+t^2v]*[(s^2m+t^2n)1+s^2u+t^2v]\\
& = & (s^2m+t^2n)^21+2(s^2m+t^2n)s^2u+s^4u^2+2(s^2m+t^2n)t^2v+t^4v^2\\
& = & (s^2m+t^2n)^21+2(s^2m+t^2n)s^2u+s^4[(a^2*a^2)-m^21-2mu]\\
&& +2(s^2m+t^2n)t^2v+t^4[(b^2*b^2)-n^21-2nv]\\
& = & s^4(a^2*a^2)+t^4(b^2*b^2)+2s^2t^2mn1+2s^2t^2(nu+mv)\,.
\end{eqnarray*}

Suppose now that $A\circledast B$ is power-associative. Then, in particular,
the subrings $A$ and $B$ are also power-associative, i.e., the identity $a^2*
a^2=(a^2*a)*a=a^3*a$ holds in $A$ and the identity $b^2*b^2=(b^2*b)*b
=b^3*b$ holds in $B$. Consequently, we obtain from the identities for $x^2
*x^2$ and $(x^2*x)*x$ that
$$0=x^2*x^2-(x^2*x)*x=s^2t^2(nu+mv)-st^3n'a-s^3tm'b$$
for any integers $s$ and $t$. In particular, this holds for $(s,t)=(1,1)$
and $(s,t)=(1,-1)$, and therefore we have that $nu+mv=\pm(n'a+m'b)$,
or $2(n'a+m'b)=0$. Hence, we obtain that $nu+mv=n'a+m'b=0$, or equivalently,
$$\sum_{\tilde{a}\in\widetilde{A}}(nm_{\tilde{a}})\tilde{a}+\sum_{\tilde{b}
\in\widetilde{B}}(mn_{\tilde{b}})\tilde{b}=0\,.$$
Since $\tilde{A}\cup\tilde{B}$ is linearly independent in $A\circledast B$,
all these coefficients must vanish. In particular, we conclude that
$nm_{\tilde{a}}=0$ for every element $\tilde{a}\in\widetilde{A}$ which
contradicts the hypothesis that $n\ne 0$ and at least one of the integers
$m_{\tilde{a}}$ is non-zero.
\end{proof}

\begin{rem}\label{gen}
\emph{Lemma \ref{criterion} remains true when instead of $a\in\tilde{A}$
we just assume that $a$ is an integer linear combination of elements of
$\tilde{A}$ and instead of $b\in\tilde{B}$ we just assume that $b$ is an
integer linear combination of elements of $\tilde{B}$.}
\end{rem}

Note that Theorem \ref{nonpowasso} corrects Theorem 5.13 and
Corollary 5.17 in \cite{FW}.\footnote{The mistake in the proof of
Theorem 5.13 was not to take into account that in the computation
of $u^2\cdot v$ or $u\cdot(u\cdot v)$, respectively, some of the terms
$[F_{\lambda'+\lambda''}^s]$ etc.\ in $u^2$ or $u\cdot v$, respectively,
are equal to $[F_0^s]$ etc., and therefore do not cancel when multiplied
by $v$ or $u$, respectively.}  In addition,
at the same time it generalizes Proposition 5.10 as well as Theorem 5.19,
Proposition 5.21, and also Example 5.14 in \cite{FW}.

\begin{thm}\label{nonpowasso}
Let $\lf$ be a non-zero Leibniz algebra over an arbitrary field $\fbb$
whose canonical Lie algebra $\lf_\lie$ is finite dimensional. Then the
Grothendieck ring $\K^\bi(\lf)$ is not power-associative. Moreover,
$\K^\bi(\lf)$ is neither alternative nor a Jordan ring.
\end{thm}

\begin{proof}
We begin by assuming that the canonical Lie algebra $\gf:=\lf_\lie$
is solvable. Since by hypothesis $\lf$ is non-zero, it follows from
\cite[Proposition 2.20]{F1} that $\gf$ is also non-zero, and therefore
$\gf$ is not perfect. Now choose a basis $\{e_1,\dots,e_m\}$ of the
derived subalgebra $[\gf,\gf]$ of $\gf$ and complete this basis to
a basis $\{e_1,\dots,e_m,h_1,\dots,h_n\}$ of $\gf$. As $\gf$ is not
perfect, we have that $n\ge 1$, and thus one can define a linear
form $\lambda$ on $\gf$ by setting $\lambda(h_1):=1$ and letting
it be zero on all other basis elements. Then $F_\lambda:=\fbb
1_\lambda$ with $\gf$-action defined by $x\cdot 1_\lambda:=
\lambda(x)1_\lambda$ is a non-trivial irreducible $\gf$-module.

In the remainder of this part of the proof we now distinguish
three cases:

Firstly, assume that the characteristic of the ground field is not
two, and set $a:=[F_\lambda^s]$ and $b:=[F_\lambda^a]+
[F_{-\lambda}^a]$. Since $\pm\lambda \ne 0$, and in this case
$\pm 2\lambda\ne 0$, we obtain that
$$a^2=[F_\lambda^s]\cdot[F_\lambda^s]=[F_\lambda^s\ovo
F_\lambda^s]=[F_\lambda^s\otimes F_\lambda^s]=[(F_\lambda
\otimes F_\lambda)^s]=[F_{2\lambda}^s]$$
and
$$b^2=[F_\lambda^a]^2+[F_{-\lambda}^a]^2+2[F_0^{s/a}]=
2[F_0^{s/a}]+[F_{2\lambda}^a]+[F_{-2\lambda}^a]$$
satisfy the hypothesis of Lemma \ref{criterion} (see Remark \ref{gen}).
Hence, the Grothendieck ring $\K^\bi(\lf)$ is not power-associative.

Secondly, assume that the ground field $\fbb$ has characteristic two
and contains more than two elements. Then one can choose an element
$c\in\fbb\setminus\{0,1\}$. Furthermore, set $a:=[F_\lambda^s] +
[F_{c\lambda}^s]$ and $b:=[F_\lambda^a]$. Since
by the choice of $c$ we have that $c\lambda\ne 0$ and $(1+c)\lambda
\ne 0$, we obtain that
$$a^2=[F_\lambda^s]^2+[F_{c\lambda}^s]^2+2[F_{(1+c)\lambda}^s]=
2[F_0^{s/a}]+2[F_{(1+c)\lambda}^s]$$
and
$$b^2=[F_\lambda^a]^2=[F_0^{s/a}]$$
satisfy the hypothesis of Lemma \ref{criterion} (see Remark \ref{gen}).
Consequently, the Grothendieck ring $\K^\bi(\lf)$ is not power-associative.

Thirdly, assume that the ground field $\fbb$ contains only two elements. In this
case consider the two-dimensional $\gf$-module $M=\fbb m_1\oplus
\fbb m_2$ on which $h_1$ acts via $h_1\cdot m_1:=m_2$, $h_1\cdot
m_2:=m_1+m_2$, and all the other basis vectors of $\gf$ act trivially.
Note that a two-dimensional module is irreducible if, and only if, it
does not contain a one-dimensional submodule, and therefore $M$
is irreducible exactly when the linear operator $h_1$ has no eigenvalue
on $M$. But the latter is the case because its characteristic polynomial
$X^2+X+1$ has no roots in $\fbb$, and thus $M$ is a non-trivial
irreducible $\gf$-module. Moreover, it is straightforward to verify that
$$M\otimes M=\fbb(m_1\otimes m_1)\oplus\fbb(m_1\otimes m_2)
\oplus\fbb(m_2\otimes m_1)\oplus\fbb(m_2\otimes m_2)$$
has the following action of $h_1$:
\begin{eqnarray*}
h_1\cdot(m_1\otimes m_1) & = & m_1\otimes m_2+m_2\otimes m_1\,,\\
h_1\cdot(m_1\otimes m_2) & = & m_1\otimes m_1+m_1\otimes m_2
+m_2\otimes m_2\,,\\
h_1\cdot(m_2\otimes m_1) & = & m_1\otimes m_1+m_2\otimes m_1
+m_2\otimes m_2\,,\\
h_1\cdot(m_2\otimes m_2) & = & m_1\otimes m_2+m_2\otimes m_1\,.
\end{eqnarray*}
In particular, for dimension reasons, we obtain from
$$h_1\cdot(m_1\otimes m_2+m_2\otimes m_1)=m_1\otimes m_2+m_2\otimes m_1$$
that
$$U:=\fbb(m_1\otimes m_1)\oplus\fbb(m_1\otimes m_2+m_2\otimes m_1)
\oplus\fbb(m_2\otimes m_2)$$
is a maximal $\gf$-submodule of $M\otimes M$,
$$V:=\fbb(m_1\otimes m_2+m_2\otimes m_1)\oplus\fbb(m_2\otimes m_2)$$
is a maximal $\gf$-submodule of $U$, and
$$W:=\fbb(m_1\otimes m_2+m_2\otimes m_1)=\gf U=\gf V$$
is a maximal $\gf$-submodule of $V$. Consequently, we have that
$$M\otimes M\supset U\supset V\supset W\supset 0$$
is a composition series of $M\otimes M$ with composition factors $(M
\otimes M)/U\cong F_\lambda$, $U/V\cong F_0$, $V/W\cong F_0$, and $W
\cong F_\lambda$, where $F_0$ denotes the one-dimensional trivial $\gf$-module
and $F_\lambda$ is the $\gf$-module defined earlier. Hence, we obtain that
$$[M\otimes M]=2[F_0]+2[F_\lambda]\,.$$
Now, by setting $a:=[M^s]$ and $b:=[M^a]$, we conclude from Lemma
\ref{criterion} that the Grothendieck ring $\K^\bi(\lf)$ is not power-associative.

Next, we assume that the canonical Lie algebra $\gf$ is not solvable. Now
consider a composition series
$$\gf_\ad=M_\ell\supset M_{\ell-1}\supset\cdots\supset M_1\supset M_0=0$$
of the adjoint $\gf$-module $\gf_\ad$, and let
$$\nf:=\bigcap_{k=1}^\ell\Ker(\rho_k)$$
be the intersection of the annihilators of the composition factors of
$\gf_\ad$, where $\rho_k$ denotes the representation of $\gf$ on the
composition factor $M_k/M_{k-1}$ $(1\le k\le\ell)$. Then it is well known
that $\nf$ is a nilpotent ideal of $\gf$ (for example, see \cite[Proposition~1]{F0}).
Furthermore, the Lie algebra $\gf/\nf$ embeds into the finite direct
product of the images $\rho_k(\gf)$ $(1\le k\le\ell)$. Now suppose that
all these images are nilpotent Lie algebras. Then their finite direct product
would be nilpotent, and thus $\gf/\nf$ would be solvable. But as $\nf$
is solvable, $\gf$ would also be solvable which contradicts our assumption.

Hence, there exists $k\in\{1,\dots,\ell\}$ for which $\rho_k(\gf)$ is
not a nilpotent Lie algebra. In particular, $N:=M_k/M_{k-1}$ is a
non-trivial irreducible $\gf$-module, and then the linear dual $N^*$
of $N$ is also a non-trivial irreducible $\gf$-module. Note that $\rho_k$
is not only a homomorphism of Lie algebras, but also a homomorphism
of $\gf$-modules. In particular, $\rho_k(\gf)$ is a $\gf$-submodule of
$\End_\fbb(N)$. Since $\End_\fbb(N)$ and $N^*\otimes N$ are isomorphic
as $\gf$-modules, every composition factor of $\rho_k(\gf)$ is also a
composition factor of $N^*\otimes N$.

Suppose now that every composition factor of $\rho_k(\gf)$ is trivial.
Then every element of $\gf$ would act nilpotently on $\rho_k(\gf)$.
But since the $\gf$-action on $\rho_k(\gf)$ coincides with the adjoint
action of $\rho_k(\gf)$ on itself, Engel's theorem (see \cite[Theorem
3.2]{H}) would imply that $\rho_k(\gf)$ is nilpotent which is a contradiction.
Hence, $\rho_k(\gf)$ has a non-trivial composition factor, and therefore
the same holds for $N^*\otimes N$.

Finally, the evaluation map from $N^*\otimes N$ into the ground field
is non-zero, and thus it is a surjective homomorphism of $\gf$-modules
from $N^*\otimes N$ onto the one-dimensional trivial $\gf$-module $F_0$.
This shows that the trivial irreducible $\gf$-module is a composition factor
of $N^*\otimes N$.

If we set $a:=[(N^*)^s]+[N^s]$ and $b:=[(N^*)^a]+[N^a]$, then
$$a^2=[(N^*)^s]^2+[N^s]^2+2[(N^*\otimes N)^s]$$
and
$$b^2=[(N^*)^a]^2+[N^a]^2+2[(N^*\otimes N)^a]$$
satisfy the hypothesis of Lemma \ref{criterion} (see Remark \ref{gen}),
and therefore the Grothendieck ring $\K^\bi(\lf)$ is not power-associative.

By virtue of \cite[Lemma 5.15]{FW}, in all the above cases, $\K^\bi (\lf)$ is
also neither alternative nor a Jordan ring which then completes the proof.
\end{proof}

\begin{rem}
\emph{The hypothesis in Theorem \ref{nonpowasso} that the canonical
Lie algebra $\lf_\lie$ of a Leibniz algebra $\lf$ is finite dimensional does
not imply that $\lf$ itself is finite dimensional.}

\emph{Namely, let $Z(\lambda)$ denote the (infinite-dimensional) Verma
module of highest weight $\lambda$ over the three-dimensional simple Lie
algebra $\slf_2(\cbb)$ of traceless complex $2\times 2$ matrices and
assume that $\lambda+1$ is not a non-negative integer. Then it follows
from Exercise 7\,(c) in Section 7 of \cite{H} and Theorem 2.3 in \cite{F2} that the
hemi-semidirect product $\lf:=\slf_2(\cbb)\ltimes_\hemi Z(\lambda)$ is an
infinite-dimensional simple (and thus non-solvable) Leibniz algebra whose canonical Lie algebra
is three-dimensional (see also Example 5.12  in \cite{FW}). Moreover, the free Leibniz algebra on one generator
is an infinite-dimensional solvable Leibniz algebra whose canonical Lie algebra
is one-dimensional (see \cite[Example 1.2\,(d)]{LP} for the definition and
\cite[Lemma 1.3]{LP} for the construction of a free Leibniz algebra).}
\end{rem}

In summary, Theorem \ref{nonpowasso} shows that we now have
a complete picture of the associativity properties (namely, alternativity,
power-associativity, or being a Jordan ring) of the Grothendieck ring
of a non-zero Leibniz algebra whose canonical Lie algebra is finite
dimensional. Namely, in this case the Grothendieck ring is highly
non-associative. Of course, the Grothendieck ring of the zero algebra
is just the ring of integers (see the remark after Proposition 5.2 in
\cite{FW}). On the other hand, at this moment one can only speculate
what happens in general as our proof of Theorem \ref{nonpowasso}
decisively uses that the canonical Lie algebra is finite dimensional.

Finally, let us take the opportunity to point out to the reader of our paper
\cite{FW} to make the following three changes:
\begin{itemize}
\item The sentence ``The proof of (b) is very similar to the proof of (a) and is
          left to the reader.'' in the proof of Lemma 5.3\,(b) should be replaced
          by ``(b) follows from (a) by setting $a':=a$.''
\item In the proof of Theorem 5.19 of \cite{ FW} we implicitly used the bijection
          between the set of the isomorphism classes of the finite-dimensional
          irreducible $\gf$-modules and the set of the dominant integral weights
          with respect to a Cartan subalgebra of $\gf$ which only works when the
          ground field is algebraically closed. But indeed, all that is needed here,
          is Weyl's theorem on complete reducibility. So the first seven lines of
          the proof of Theorem 5.19 in \cite{FW} should be replaced by
          ``Furthermore, it follows from Weyl's theorem on complete reducibility
          (see [13, Theorem 6.3]) that
          $$M\otimes M=m_0L_0\oplus\bigoplus\limits_{i\in I}m_i L_i\,,$$
          where $L_0$ denotes the trivial irreducible $\gf$-module, $\{[L_i]\mid i
          \in I\}$ is the set of the isomorphism classes of finite-dimensional non-trivial
          irreducible $\gf$-modules, and $m_i$ denotes the multiplicity of $L_i$ in
          $M\otimes M$.''
\item In Remark 5.22, the somewhat sloppy statement regarding the existence
          of non-degenerate invariant bilinear forms on graded Lie algebras of
          Cartan type should be changed to ``\dots, for {\bf some of the} graded
          Lie algebras of Cartan type (see Section 6 in Chapter 4 of [22]),
          \dots`` Namely, except for the Hamiltonian algebras, the other three
          families admit only in a very few cases such bilinear forms, and therefore
          our original statement might have been confusing.
\end{itemize}
\medskip


\noindent {\bf Acknowledgment.} The first author would like to thank Layla
Sorkatti and Kenneth Price for inviting him to speak in the Special Session
on Non-Associative Rings and Algebras at the Spring Southeastern Sectional
AMS Meeting in March 2026. In particular, he is very grateful to the organizers
and to the AMS for allowing him to present the results of \cite{FW} at this
meeting in a virtual talk.




\begin{thebibliography}{99}



\bibitem{BMS}
M. Bremner, L. Murakami, and I. Shestakov:
Chapter 86: Nonassociative Algebras,
in: {\it Handbook of Linear Algebra\/} (Second enlarged edition) (ed. L. Hogben),
Discrete Mathematics and its Applications, Chapman \& Hall/CRC Press, Boca Raton, FL, 2014.

\bibitem{F0}
J. Feldvoss:
Chief factors and the principal block of a restricted Lie algebra,
in: The Monster and Lie Algebras, Columbus, OH, 1996
(eds.\ J. C. Ferrar and K. Harada),
{\it Ohio State Univ.\ Math.\ Res.\ Inst.\ Publ.\/}, vol.\ {\bf 7},
Walter de Gruyter, Berlin, 1998, pp.\ 187--194.

\bibitem{F1}
J. Feldvoss:
Leibniz algebras as non-associative algebras,
in: Nonassociative Mathematics and Its Applications, Denver, CO, 2017
(eds. P. Vojt\v{e}chovsk\'y, M. R. Bremner, J. S. Carter, A. B. Evans,
J. Huerta, M. K. Kinyon, G. E. Moorhouse, J. D. H. Smith),
{\it Contemp. Math.\/}, vol. {\bf 721},
American Mathematical Society, Providence, RI, 2019, pp. 115--149.

\bibitem{F2}
J. Feldvoss:
Semi-simple Leibniz algebras I,
Preprint (16 pages), available as arXiv:2401.05588 [math.RA], 2024.

\bibitem{FW}
J. Feldvoss and F. Wagemann:
Tensor products of Leibniz bimodules and Grothendieck rings,
{\it J. Non-Associative Structures\/} {\bf 1} (2026), no.\ 1, Paper no.\ 4 (45 pages).

\bibitem{H}
J. E. Humphreys:
{\it Introduction to Lie Algebras and Representation Theory\/} (Third printing, revised),
Graduate Texts in Mathematics, vol. {\bf 9},
Springer-Verlag, New York/Berlin, 1980.

\bibitem{LP}
J.-L. Loday and T. Pirashvili:
Universal enveloping algebras of Leibniz algebras and (co)homology,
{\it Math. Ann.\/} {\bf 296} (1993), no. 1, 139--158.

\bibitem{S}
R. D. Schafer:
{\it An Introduction to Nonassociative Algebras\/}
(Slightly corrected reprint of the 1966 original),
Dover Publications, Mineola, NY, 2017.


\end{thebibliography}
\end{document}